\documentclass[10pt, a4paper]{amsart}

\usepackage[utf8]{inputenc}
\usepackage[english]{babel}

\usepackage{amsthm}
\usepackage{amstext}
\usepackage{amsmath}
\usepackage{amsfonts}
\usepackage{amssymb}
\usepackage{mathtools}
\usepackage{bm}

\usepackage{tikz}
\usepackage{graphics}
\usepackage{wrapfig}
\usetikzlibrary{matrix,arrows,decorations.pathmorphing}

\usepackage[a4paper]{geometry}
\usepackage{tikz-cd}
\usepackage{epigraph}

\usepackage{hyperref}
\usepackage{cite}

\usepackage{enumerate}
\usepackage{enumitem}

\author{Matteo Tamiozzo}
\title{The tame Hilbert symbol via K-theory and central extensions}

\swapnumbers
\newtheorem{teo}[subsection]{Theorem}     
\theoremstyle{plain}                    
\newtheorem{prop}[subsection]{Proposition}    
\newtheorem{corol}[subsection]{Corollary}     
\newtheorem{lem}[subsection]{Lemma}         
\theoremstyle{definition}               
\newtheorem{defin}[subsection]{Definition}
\theoremstyle{remark}                   
\newtheorem{rem}[subsection]{Remark}      
\newtheorem{notat}[subsection]{Notation}

\newcommand{\Addr}{{\bigskip
\footnotesize
\textsc{Department of Mathematics, Imperial College London, London SW7 2AZ, UK}

\textit{Email address:} \texttt{m.tamiozzo@imperial.ac.uk}
}}

\numberwithin{equation}{subsection}

\begin{document}

\date{}

\begin{abstract}
Let $K$ be a mixed characteristic local field whose residue field has cardinality $q$, and let $n$ be an integer dividing $q-1$. In the first part of this document, inspired by \cite{cla12}, we construct a $K$-theoretic enhancement of the $n$-th power residue symbol $K^\times \times K^\times \rightarrow \bm{\mu}_n$. In the second part we construct central extensions of $GL_m(K)$ by $\bm{\mu}_n$ and we express the $n$-th power residue symbol in terms of a symbol defined using the extension obtained for $m=1$, generalising \cite{adck89, pr08}. Our constructions both rely on the study of finite free pointed $\bm{\mu}_n$-sets.
\end{abstract}

\maketitle

\tableofcontents

\section{Introduction}
\subsection{The tame symbol in equal characteristic} Let $k$ be a finite field and $X$ a smooth, projective curve over $k$. We will denote by $|X|$ the set of closed points of $X$. For $x\in|X|$ we let $\mathcal{O}_{x}$ be the completed local ring at $x$, with maximal ideal $\mathrm{m}_{x}$ and residue field $k(x)$; let $K_{x}$ be the fraction field of $\mathcal{O}_{x}$ and $v_{x}$ the valuation on $K_{x}$. The field extension $k(x)/k$ is finite, hence we have a norm map $N_{k(x)/k}: k(x)^{\times}\rightarrow k^{\times}$. The tame residue symbol at $x$ is the map
\begin{align}\label{eq-hilbsym}
(\cdot, \cdot)_{x}:K_{x}^{\times}\times K_{x}^{\times} & \rightarrow k^{\times}\\
\nonumber (f, g) & \mapsto N_{k(x)/k}\left((-1)^{v_x(f)v_x(g)}\frac{f^{v_x(g)}}{g^{v_x(f)}} \pmod{\mathrm{m}_x} \right);
\end{align}
notice that $\displaystyle \frac{f^{v_{x}(g)}}{g^{v_{x}(f)}}$ has valuation zero, hence its reduction modulo $\mathrm{m}_{x}$ makes sense and is a unit, so that $(f, g)_{x}$ does indeed belong to $k^{\times}$. For non-zero elements $f, g$ in the function field of $X$ the symbol $(f, g)_{x}$ equals one for all but finitely many $x\in |X|$; Weil's reciprocity law states that
$$
\prod_{x\in|X|}(f, g)_{x}=1.
$$

\subsection{} One can look at the symbol $(\cdot, \cdot)_{x}$ from (at least) two different points of view: on the one hand it is a Steinberg symbol, hence it factors through a map $K_{2}(K_{x})\rightarrow k^{\times}$, which can be described purely in $K$-theoretic terms, as we will later recall. Better, one can define a family of maps $K_{i}(K_{x})\rightarrow K_{i-1}(k)$ for $i\geq 1$, specialising to the tame symbol for $i=2$. The reciprocity law holds true for every $i$ \cite{cla20}; for example, for $i=1$ it recovers the product formula.

On the other hand, Arbarello, De Concini, Kac \cite{adck89} and Pablos Romo \cite{pr02} gave a different construction of (a symbol which is very close to) the tame symbol: in a nutshell, they consider the group $GL(K_{x},\ \mathcal{O}_{x})$ of $k$-linear automorphisms $f : K_{x}\rightarrow K_{x}$ satisfying the additional assumption that $f(\mathcal{O}_{x})$ and $\mathcal{O}_{x}$ are commensurable, i. e. $(f(\mathcal{O}_{x})+\mathcal{O}_{x})/(\mathcal{O}_{x}\cap f(\mathcal{O}_{x}))$ is a finite dimensional $k$-vector space. For such an $f$ one can define the determinant of $f(\mathcal{O}_{x})$ {\it relative to} $\mathcal{O}_{x}$; it is a one-dimensional $k$-vector space, denoted by $\det(\mathcal{O}_{x}\mid f(\mathcal{O}_{x}))$. For $f,\ g\in GL(K_{x},\ \mathcal{O}_{x})$ there is a canonical isomorphism $\det(\mathcal{O}_{x}\mid f(\mathcal{O}_{x}))\otimes_{k}\det(\mathcal{O}_{x}\mid g(\mathcal{O}_{x}))\rightarrow\det(\mathcal{O}_{x}\mid fg(\mathcal{O}_{x}))$, allowing to define a group structure on the set $\widetilde{GL}(K_{x},\ \mathcal{O}_{x})=\{(f,\ s),\ f\in GL(K_{x},\ \mathcal{O}_{x}),\ s\in\det(\mathcal{O}_{x}\mid f(\mathcal{O}_{x}))\smallsetminus \{0\}\}$. One obtains a central extension of groups
\begin{equation}\label{eqchar-metapl}
1\rightarrow k^{\times}\rightarrow\widetilde{GL}(K_{x},\ \mathcal{O}_{x})\rightarrow GL(K_{x},\ \mathcal{O}_{x})\rightarrow 1.
\end{equation}
Now any $f, g\in K_{x}^{\times}$ give rise to commuting elements in $GL(K_{x},\ \mathcal{O}_{x})$; it follows that, chosen arbitrary lifts $\tilde{f},\tilde{g}\in \widetilde{GL}(K_{x},\ \mathcal{O}_{x})$ of $f, g$, the commutator $[\tilde{f},\tilde{g}]=\tilde{f}\tilde{g}\tilde{f}^{-1}\tilde{g}^{-1}$ belongs to $k^{\times}$ and is independent of the choice of the lifts. We therefore obtain a well-defined element $[\tilde{f},\tilde{g}]\in k^{\times}$ which can be shown to be equal, up to a sign, to the symbol $(f, g)_{x}$ \cite[Proposition 5.1]{pr02}. Among other things, this point of view on the tame symbol can be used to prove Weil's reciprocity law, as first noticed in \cite{adck89}; see also \cite{apr04}. Interestingly, Clausen's $K$-theoretic approach to reciprocity mentioned above (see also \cite{cla17}) seems to be at least vaguely reminiscent of some arguments in \cite{adck89}, \cite{apr04}.

\subsection{The tame symbol in mixed characteristic.} In this document we are interested in tame Hilbert symbols for local fields of mixed characteristic. Let $K$ be a local field of characteristic $0$ and residue characteristic $p$; let $\mathcal{O}$ be the ring of integers of $K$, with uniformiser $\pi$ and residue field $k=\mathcal{O}/\pi$ of cardinality $q$. We denote by $v$ the valuation on $K$.

For an integer $n\mid q-1$ the group $\bm{\mu}_n$ of $n$-th roots of unity is contained in $K^{\times}$. Reduction modulo $\pi$ identifies $\bm{\mu}_{n}$ with the group of $n$-th roots of unity in $k$. The $n$-th power residue symbol, analogous to the symbol \eqref{eq-hilbsym} in this setting, is defined as
\begin{align*}
(\cdot, \cdot)_{n}& :K^{\times}\times K^{\times}\rightarrow\bm{\mu}_{n}\\
(a, b)& \mapsto \left((-1)^{v(a)v(b)}\frac{a^{v(b)}}{b^{v(a)}}\ \pmod{\pi}\right)^{\frac{q-1}{n}}.
\end{align*}
The definition resembles the one in the equal characteristic case, except that the target is not any more the group of units of a field, and the norm map is replaced by the $\displaystyle \frac{q-1}{n}$-power map. This makes it a priori unclear how to upgrade the power residue symbol to a $K$-theoretic map. On another note, the second construction of the tame symbol in equal characteristic discussed above relies on the fact that in that setting the relevant field $K_{x}$ is a vector space over the field $k$; no such a base field is available in mixed characteristic.

\subsection{} The aim of this note is to explain why all of the previous phenomena in mixed characteristic are an instance of the same problem, and can be overcome via the same idea, somewhat restoring the symmetry between the equal and unequal characteristic situation. Morally, the idea - which is by no means ours and has already been suggested several times in the literature - is to regard $\bm{\mu}_{n}$ as a sort of ``base field'' over which $K$ lives; while this does not literally make sense, it suggests to replace vector spaces with (pointed) finite sets with a free action of $\bm{\mu}_{n}$. These form a perfectly reasonable category whose $K$-groups one can consider, and we will show that these groups are the target of a family of maps from $K_{i}(K)$ recovering the power residue symbol for $i=2$. Modulo replacing vector spaces with $\bm{\mu}_{n}$-sets, our $K$-theoretic construction will mirror very closely its equal characteristic counterpart.

Our $K$-theoretic investigations will also suggest a natural definition of the notion of {\it determinant} of a finite set with free $\bm{\mu}_{n}$-action; we will then show that it enjoys formal properties analogous to those of the usual determinant, and that the resulting theory is strong enough to allow us to carry over the arguments in \cite{adck89} and construct central extensions
\begin{equation}\label{mixchar-metapl}
1 \rightarrow \bm{\mu}_{n}\rightarrow\widetilde{GL}_{m}(K)\rightarrow GL_{m}(K)\rightarrow 1,\ m\geq 1.
\end{equation}
As in the equal characteristic situation, this will allow us to define a symbol which for $m=1$ equals, up to sign, the {\it n}-th power residue symbol.

\subsection{} Let us point out that our definition of determinant of a $\bm{\mu}_{n}$-set, and its relation with the {\it n}-th power residue symbol, are not new: as far as we know this approach was first proposed by Kapranov and Smirnov \cite{ks}; some properties of the determinant which we will need are also stated in the manuscript \cite{ks}, but neither the relation of this notion with $K$-theory nor its application to the construction of the central extensions described above are explored there. In the body of the text we will provide proofs of the properties needed to construct the extensions \eqref{mixchar-metapl}.

\subsection{} In \cite{pr08} the extensions \eqref{mixchar-metapl} are constructed in a different way under the additional assumption that $n+1$ equals the cardinality of the residue field of $K$. At the end of \cite{pr08} Pablos Romo suggests to try to adapt the approach of \cite{adck89} to the study of reciprocity laws over number fields. Our construction may be regarded as a first step in this direction; however let us stress that we only consider tame symbols throughout this note, hence we unfortunately cannot address reciprocity laws yet.

\subsection{} This text initially grew out of the author's attempt to understand Clausen's construction of the tame quadratic residue symbol in \cite{cla12}. Although in the end our $K$-theoretic construction of the {\it n}-th power residue symbol for $n=2$ appears to be different from Clausen's one (to the best of our understanding), we will start this document by providing a low-tech exposition of part of Clausen's construction in the tame case, and a detailed explanation of why it recovers the quadratic residue symbol. We hope that this can be of some use to those readers wishing to approach \cite{cla12}.

\subsection{Acknowledgements.} We would like to thank Dustin Clausen for answering our (naive) questions and sharing his notes \cite{cla20}. We also wish to thank Chirantan Chowdhury, Lorenzo Mantovani and Jan Nekov\'{a}\v{r} for helpful discussions related to the topic of this paper. The author's research is supported by the ERC Grant 804176.

\section{The quadratic residue symbol}

\subsection{} Fix a mixed characteristic local field $K$ with residue field $k$ of odd cardinality $q$. We will now describe Clausen's $K$-theoretic enhancement of the Hilbert symbol $K^{\times}\times K^{\times}\rightarrow\{\pm 1\}$. This section should be regarded as an expanded version of the footnote in \cite[page 3]{cla12}. Our treatment will rely on Quillen's classical definition of $K$-theory via the $+$ construction, which we recall first.

\subsection{Reminder of Quillen's $+$ construction and $K$-theory of rings.} Let $G$ be a group and $N$ a normal, perfect subgroup of $G$. The classifying space $BG$ has fundamental group $\pi_{1}(BG)=G$, hence $N$ is a perfect, normal subgroup of $\pi_{1}(BG)$. In this situation Quillen's $+$ construction gives rise to a CW-complex $BG^{+}$ with a map $\iota$ : $BG\rightarrow BG^{+}$ sending $N\subset\pi_{1}(BG)$ to the identity, and initial (up to homotopy) among all maps $BG\rightarrow Y$ killing $N$. The map $\iota$ induces an isomorphism $\pi_{1}(BG)/N=G/N\simeq\pi_{1}(BG^{+})$.

In particular, let $R$ be a (commutative) ring and let $GL(R)$ be the direct limit of the general linear groups $GL_{r}(R)$ along the transition maps adding 1 to the bottom right corner of a matrix. The commutator subgroup $[GL(R),\ GL(R)]\subset GL(R)$ is perfect. This follows from Whitehead's theorem, identifying $[GL(R),\ GL(R)]$ with the group generated by elementary matrices, which is perfect. Letting $BGL(R)^{+}$ be the space obtained performing the $+$ construction on $BGL(R)$ with respect to $N=[GL(R),\ GL(R)]$, the $K$-groups of $R$ are defined, for $i\geq 1$, as $K_{i}(R)= \pi_{i}(BGL(R)^{+})$ .

In particular we have $K_{1}(R)=GL(R)^{ab}$ by construction. The determinant maps on each $GL_{r}(R)$ are compatible with the transition maps, hence they induce a surjection $K_{1}(R)\rightarrow R^{\times}$, fitting into a split short exact sequence
\begin{equation*}
1\rightarrow SK_{1}(R)\rightarrow K_{1}(R)\rightarrow R^{\times}\rightarrow 1;
\end{equation*}
the {\it special Whitehead group} $SK_{1}(R)$ is trivial if $R$ is a field, hence in this case we have $K_{1}(R)=R^{\times}$.

\subsection{$K$-theory of pointed $q$-sets.} Recall that $k$ is the finite field with $q$ elements. For each integer $r\geq 0$ let $\Sigma_{q^{r}}^{*}$ be the set of permutations of the {\it pointed} set $(k^{r},\ 0)$:

\begin{equation*}
\Sigma_{q^{r}}^{*}=\{\sigma: k^{r}\rightarrow k^{r} \text{ bijective}, \ \sigma(0)=0 \}.
\end{equation*}

For every $r\geq 0$ we have a map $i_{r}: \Sigma_{q^{r}}^{*}\rightarrow\Sigma_{q^{r+1}}^{*}$ sending a permutation $\sigma$ to $\sigma\times Id$ : $k^{r}\times k\rightarrow k^{r}\times k$. Let $\Sigma_{q^{\infty}}^{*}$ be the direct limit of the groups $\Sigma_{q^{r}}^{*}$ with respect to the transition maps $i_{r}$. We claim that the commutator subgroup $[\Sigma_{q^{\infty}}^{*},\ \Sigma_{q^{\infty}}^{*}]$ is perfect. Indeed $\Sigma_{q^{r}}^{*}$ is isomorphic to the symmetric group on $q^{r}-1$ letters, whose commutator subgroup is the alternating group on $q^{r}-1$ letters. This is simple, hence perfect, if $q^{r}-1\geq 5$; the claim follows.

We can therefore look at the space $B(\Sigma_{q^{\infty}}^{*})^{+}$, where the plus construction is performed with respect to $[\Sigma_{q^{\infty}}^{*},\ \Sigma_{q^{\infty}}^{*}]$. We define
\begin{equation*}
K_{i}(\mathcal{S}_{q^{\infty}}^{*})=\pi_{i}(B(\Sigma_{q^{\infty}}^{*})^{+}),\ i\geq 1.
\end{equation*}

\begin{lem}
The sign maps $sgn: \Sigma_{q^{r}}^{*}\rightarrow\{\pm 1\}$ are compatible with the transition maps $i_{r}$ and induce an isomorphism $K_{1}(\mathcal{S}_{q^{\infty}}^{*})\simeq \mathbf{Z}/2\mathbf{Z}$.
\end{lem}

\begin{proof}
Given $r\geq 0$ and $\sigma\in\Sigma_{q^{r}}^{*}$, the permutation $\sigma\times Id$ of $k^{r}\times k$ stabilises each $k^{r}\times\{x\}, \ x\in k$, and acts on it as $\sigma$. It follows that the sign of $\sigma\times Id$ equals $sgn(\sigma)^{q}=sgn(\sigma)$, where the equality holds as $q$ is odd. Hence the sign maps yield a well defined, surjective map $\Sigma_{q^{\infty}}^{*}\rightarrow\{\pm 1\}$, whose kernel is the commutator subgroup (which coincides with the limit of the alternating groups).
\end{proof}

\subsection{The quadratic $K$-residue symbol.} For each $r\geq 0$ we have $GL_{r}(k)=Aut_{k}(k^{r})$, and the map $GL_{r}(k)\rightarrow GL_{r+1}(k)$ adding 1 in the bottom right corner of a matrix is the map sending an automorphism $f\in Aut_{k}(k^{r})$ to $f\times Id\in Aut_{k}(k^{r+1})$. Forgetting the $k$-vector space structure on $k^{r}$ we obtain compatible maps $GL_{r}(k)\rightarrow\Sigma_{q^{r}}^{*}$ for $r\geq 0$, hence a map
\begin{equation*}
GL(k)\rightarrow\Sigma_{q^{\infty}}^{*}.
\end{equation*}
The above maps induce, for each $i\geq 1$, a morphism
\begin{equation*}
q_{i}:K_{i}(k)\rightarrow K_{i}(\mathcal{S}_{q^{\infty}}^{*}).
\end{equation*}

\begin{prop}\label{legendre}
The map $q_{1}$: $k^{\times}\rightarrow\{\pm 1\}$ coincides with the Legendre symbol, i. e. it sends $x$ to $x^{\frac{q-1}{2}}$.
\end{prop}

\begin{proof}
For every $r\geq 0$, let $f_{r}: GL_{r}(k)\rightarrow\Sigma_{q^{r}}^{*}$ be the forgetful map defined above. Because the target of $sgn$ is abelian, the composite $sgn\circ f_{r}$ factors through the determinant map, yielding a morphism $q$ making the following diagram commutative
\begin{center}
\begin{tikzcd}
GL_{r}(k)\arrow[r, "f_r"] \arrow[d, "det"] & \Sigma_{q^{r}}^{*}\arrow[d, "sgn"]\\
k^\times \arrow[r, "q"] &  \{\pm 1\}.
\end{tikzcd}
\end{center}
To prove the proposition we must show that $q$ coincides with the Legendre symbol. Every matrix $M\in GL_{r}(k)$ can be written as $M=DM_{1}$, where $D=\mathrm{diag}(\det(M), 1, 1, \ldots, 1)$ and $det(M_{1})=1$; hence it suffices to show that $sgn\circ f_{r}(\mathrm{diag}(d, 1, 1, \ldots , 1)) =d^{\frac{q-1}{2}}$, for every $d\in k^{\times}$. We know by the previous lemma that the sign map is invariant with respect to the transition morphisms, hence we are reduced to prove the statement for $r=1$. In other words we must show that, for $d\in k^{\times}$, the sign of the multiplication by $d$ map $\sigma_{d}: k\rightarrow k$ is $d^{\frac{q-1}{2}}$. This is Zolotarev's lemma; to prove it we notice that, for a generator $x$ of $k^{\times}$, the map $\sigma_{x}$ cyclically permutes $k^{\times}$, hence it has sign $(-1)^{q-2}=-1$ (as $q$ is odd). Therefore the sign of $\sigma_{d}$ equals one if and only if $d$ is an even power of $x$.
\end{proof}

\subsection{Remarks.}
\begin{enumerate}
\item As we will recall in the next section, for every $i \geq 1$ there is a map $\partial_{i-1}:K_{i}(K)\rightarrow K_{i-1}(k)$; when $i = 2$ it is induced by the symbol $K^\times \times K^\times \rightarrow k^\times$ sending $(a, b)$ to $\displaystyle (-1)^{v(a)v(b)}\frac{a^{v(b)}}{b^{v(a)}}\ \pmod{\pi}$. Hence we obtain maps $q_{i-1}\circ \partial_{i-1}: K_i(K) \rightarrow K_{i-1}(\mathcal{S}_{q^{\infty}}^{*})$ for $i \geq 1$, recovering the quadratic residue symbol for $i = 2$.
\item  If $K = \mathbf{Q}_p$, the maps $q_i : K_i(\mathbf{F}_p) \rightarrow K_i(\mathcal{S}_{p^{\infty}}^{*})$ defined above coincide with those induced in $K$-theory by the map $Vect^{\sim}_{\mathbf{F}_p}\rightarrow Set^{\sim}_p$ constructed in \cite[page 8]{cla12}. As explained to the author by Clausen, the target can be identified with the spectrum of units in the $p$-inverted sphere spectrum $\mathbb{S}[1/p]$, and for suitable $i$ the image of the map $q_i$ can be related to the image of the real $J$-homomorphism.
\end{enumerate}

\section{The tame symbol via K-theory}

\subsection{}\label{tameK.1} Let $K$ be a local field of characteristic zero, with finite residue field $k$ of cardinality $q$; let $v$ be the valuation on $K$. Fix an integer $n$ dividing $q-1$; the {\it n}-th power residue symbol
\begin{align*}
(\cdot, \cdot)_{n}: \ K^{\times}\times K^{\times} & \rightarrow \bm{\mu}_{n}\\
(a, b)& \mapsto \left((-1)^{v(a)v(b)}\frac{a^{v(b)}}{b^{v(a)}}\pmod{\pi}\right)^{\frac{q-1}{n}}
\end{align*}
is a Steinberg symbol, hence it factors through a map from the Milnor $K$-group $K_{2}^{M}(K)$ to $\bm{\mu}_{n}$. In view of Matsumoto's theorem, the symbol induces a map $K_{2}(K)\rightarrow \bm{\mu}_{n}$; the aim of this section is to construct $K$-theoretic maps with source $K_{i}(K)$ for $i\geq 1$ recovering the power residue symbol for $i=2$. Similarly to what we did above, the idea is to forget the vector space structure; however, in this section we will {\it remember the action of n-th roots of unity}. As we mentioned in the introduction, those play a role somewhat analogous to that of the base field in the equal characteristic setting; in order to highlight the analogy we will first recall the $K$-theoretic interpretation of tame symbols in the latter situation.

\subsection{The boundary maps} Let $A$ be a $DVR$ with maximal ideal $\mathfrak{m}$, fraction field $F$ and residue field $k$; there is a homotopy fibre sequence $K(k)\rightarrow K(A)\rightarrow K(F)$ inducing a long exact sequence
\begin{equation*}
\ldots K_{i}(k)\rightarrow K_{i}(A)\rightarrow K_{i}(F) \xrightarrow{\partial_{i-1}} K_{i-1}(k)\rightarrow.\ldots .
\end{equation*}
If $i=2$ then the map $\partial_{1}: K_{2}(F)\rightarrow K_{1}(k)$ is induced by the symbol sending $(a, b)\in F^{\times}\times F^{\times}$ to $(-1)^{v(a)v(b)}\displaystyle \frac{a^{v(b)}}{b^{v(a)}} \pmod{\mathfrak{m}}$.

\subsection{The tame symbol in equal characteristic} Fix a smooth, projective curve $X$ over a finite field $k$ of odd characteristic, and a closed point $x$ of $X$. Let $\mathcal{O}_{x}$ be the completed local ring of $X$ at $x$, and $\mathrm{m}_{x}$ (resp. $k(x)$, resp. $K_{x}$) the maximal ideal (resp. residue field, resp. fraction field) of $\mathcal{O}_{x}$. Let $v_{x}$ be the valuation on $K_{x}$ and $N_{k(x)/k}: k(x)^{\times}\rightarrow k^{\times}$ the norm map. The tame symbol at $x$ is the map $K_{x}^\times \times K_{x}^\times \rightarrow k^{\times}$ sending $(f, g)$ to
\begin{equation*}
(f, g)_{x}=N_{k(x)/k}\left( (-1)^{v_x(f)v_x(g)}\frac{f^{v_x(g)}}{g^{v_x(f)}} \pmod{\mathrm{m}_x} \right).
\end{equation*}
As explained above, the map sending $(f, g)$ to $(-1)^{v_{x}(f)v_{x}(g)}\displaystyle \frac{f^{v_{x}(g)}}{g^{v_{x}(f)}}\pmod{\mathrm{m}_x}$ is the second of the sequence of boundary maps $\partial_{i-1} : K_{i}(K_{x})\rightarrow K_{i-1}(k(x))$. We will now argue that the norm map also has a natural $K$-theoretic enhancement, thereby allowing to define a family of maps $K_{i}(K_{x})\rightarrow K_{i-1}(k), i\geq 1$, recovering the tame symbol for $i=2$.

The forgetful functor
\begin{equation}\label{forget}
\{\text{finite dimensional } k(x)-\text{vector spaces}\} \rightarrow \{ \text{finite dimensional } k- \text{vector spaces}\}
\end{equation}
induces maps in $K$-theory $N_{i}: K_{i}(k(x))\rightarrow K_{i}(k)$ for $i\geq 1$.

\begin{lem}
The map $N_{1} : k(x)^{\times}\rightarrow k^{\times}$ coincides with the norm map.
\end{lem}

\begin{proof}
Let $d_{x}=[k(x): k]$; the natural maps $GL_{r}(k(x))\rightarrow GL_{d(x)r}(k)\rightarrow GL(k)$ induce a map $GL(k(x))\rightarrow GL(k)$, whose abelianisation is the map $N_{1}$. For every integer $r\geq 1$ the induced map $f_{r}$ : $GL_{r}(k(x))\rightarrow GL(k)^{ab}=k^{\times}$ factors through the determinant. By the same argument used in the proof of Proposition \ref{legendre} it suffices to check that $f_{r}$ sends diag $(d, 1,1, \ldots, 1)$ to $N_{k(x)/k}(d)$; since the maps $f_{r}$ are compatible as $r$ varies we are reduced to the case $r=1$, where the equality $N_{k(x)/k}=f_{1}$ holds true by definition of the norm. 
\end{proof}

\begin{rem}
In fact, Clausen discovered that the boundary maps $\partial_{i}$ can {\it also} be interpreted as coming from suitable {\it forgetful maps}, exploiting the fact that every element of $GL_{r}(K_{x})$ can be seen as an automorphism of $(K_{x})^{r}$ {\it as a locally linearly compact} $k(x)$-vector space. This leads to a very natural and elegant proof of Weil's reciprocity law \cite{cla20}.
\end{rem}

\subsection{The mixed characteristic situation.} Let us now come back to the mixed characteristic setting. With the notation introduced in \ref{tameK.1}, we are interested in the symbol $(\cdot, \cdot)_{n}: K^{\times}\times K^{\times}\rightarrow \bm{\mu}_{n}$; the induced map $K_{2}(K)\rightarrow \bm{\mu}_{n}$ is the composite of the boundary map $\partial_{1}$ : $K_{2}(K)\rightarrow K_{1}(k)$ and of the map
\begin{align*}
p_{n}:k^{\times}& \rightarrow \bm{\mu}_{n}\\
x & \mapsto x^{\frac{q-1}{n}}.
\end{align*}
In what follows we will upgrade the {\it n}-th power residue symbol to a $K$-theoretic map. To do this it suffices to deal with the map $p_{n}$, on which we will concentrate from now on; thinking of $\bm{\mu}_{n}$ as of our ``base field'', we will construct a functor
\begin{equation*}
\{\text{finite dimensional } k-\text{vector spaces}\} \rightarrow \{\text{finite dimensional } ``\bm{\mu}_{n}-\text{vector spaces''}\}
\end{equation*}
analogous to \eqref{forget}, and study its effect on $K$-theory. We will slightly change language and work with Quillen's definition of $K$-theory of symmetric monoidal categories via the $S^{-1}S$ construction.

\subsection{} Take a small symmetric monoidal category $(\mathcal{S}, \otimes, 0)$ satisfying the following conditions:
\begin{enumerate}
\item Every morphism in $\mathcal{S}$ is an isomorphism.
\item  For every couple of objects $s, t\in \mathcal{S}$, the natural map $Aut(t)\rightarrow Aut(s\otimes t)$ is injective.
\end{enumerate}

\subsection{Examples.}
\begin{enumerate}
\item Let $Fin$ be the category whose objects are the empty set and, for $n\geq 1$, the set $\underline{n}= \{1,2,\ldots, n\}$. Morphisms are bijections, and we put a monoidal structure on $Fin$ via disjoint union.
\item We denote by $\mathcal{S}_{q^{\infty}}^{*}$ the category whose objects are the pointed sets $(k^{r}, 0)$ and morphisms are bijections preserving the marked point. We put a monoidal structure on $\mathcal{S}_{q^{\infty}}^{*}$ via cartesian product of sets.
\item For a field $F$, let $V_{F}$ be the category whose objects are finite dimensional vector spaces $F^{n}, n\geq 0$, and morphisms are $F$-linear isomorphisms. The monoidal structure is given by direct sum of vector spaces.
\end{enumerate}

\subsection{$K$-theory of symmetric monoidal categories.} Quillen defined, for $(\mathcal{S}, \otimes, 0)$ as above, a ``group completion'' $S^{-1}\mathcal{S}$, whose objects are pairs $(a, b)$ of objects of $\mathcal{S}$, and a morphism from $(a, b)$ to $(c, d)$ is an equivalence class of morphisms
\begin{equation*}
(a, b) \xrightarrow{s \otimes (-)} (s\otimes a, s\otimes b)\rightarrow(c, d);
\end{equation*}
see \cite[Chapter IV, Definition 4.2]{wei13} for more details. Morally, we are thinking of $(a, b)$ as the ``quotient'' of $a$ and $b$, and the above definition generalises the familiar rule $\displaystyle \frac{as}{bs}=\frac{a}{b}$. This intuition is confirmed by the fact that the map on $\pi_{0}$ (the set of isomorphism classes of objects) induced by the functor sending $a$ to $(a, 0)$ identifies $\pi_{0}(S^{-1}\mathcal{S})$ with the group completion of the monoid $\pi_{0}(\mathcal{S})$.

One can now look at the geometric realisation $BS^{-1}\mathcal{S}$ of the small category $S^{-1}\mathcal{S}$ (i. e. the geometric realisation of its nerve), and give the following definition:
\begin{defin}{(Quillen)}
For $i\geq 0$, the $i$-th $K$-group of $\mathcal{S}$ is
\begin{equation*}
K_{i}(\mathcal{S})=\pi_{i}(S^{-1}\mathcal{S}).
\end{equation*}
\end{defin}

\begin{lem}{(\cite{wei13}, Corollary 4.8.1, p. 334)}
$K_{1}(\displaystyle \mathcal{S})=\varinjlim_{s}H_{1}(Aut(s),\ \mathbf{Z})$.
\end{lem}

\subsection{Example.}
If $\mathcal{S}=V_{F}$ then Quillen's $+=S^{-1}$ theorem asserts that $ B(S^{-1}\mathcal{S})\sim \mathbf{Z}\times BGL(F)^{+}$ - this is true for arbitrary rings - hence the above definition of $K$-theory coincides with the one recalled in the second section.

If $\mathcal{S}=Fin$ one has $B(S^{-1}\mathcal{S})\sim \mathbf{Z}\times B\Sigma_{\infty}^{+}$, where $\Sigma_{\infty}$ is the direct limit of the symmetric groups $\Sigma_{n}$ along the maps fixing the added points. The Barratt-Priddy theorem identifies $B(S^{-1}\mathcal{S})$ with the infinite loop space $\displaystyle \lim_{n \to \infty}\Omega^{n}S^{n}$ (whose homotopy groups are the stable homotopy groups of spheres).

Finally, the $K$-groups of $\mathcal{S}=\mathcal{S}_{q^{\infty}}^{*}$ also agree with those defined in the previous section. This follows from the general $+=S^{-1}$ result given in \cite[Theorem 4.10, p. 336]{wei13}.

\begin{rem}
The construction of the quadratic residue symbol given before can be reinterpreted as coming from the functor $V_{k}\rightarrow \mathcal{S}_{q^{\infty}}^{*}$ obtained sending a vector space to the underlying pointed set - i. e., we remember the identity of the underlying abelian group.
\end{rem}

\subsection{The category of finite free pointed $\bm{\mu}_{n}$-sets.} By a {\it finite free pointed} $\bm{\mu}_{n}$-set we mean a finite pointed set $(X,\ *)$ together with a left action of $\bm{\mu}_{n}$ on $X$ which fixes the distinguished point and is free when restricted to $X\smallsetminus \{*\}$. We denote by $V_{\bm{\mu}_{n}}$ the category whose objects are finite free pointed $\bm{\mu}_{n}$-sets (more precisely, we choose one representative in each isomorphism class) and morphisms are $\bm{\mu}_{n}$-equivariant bijections preserving the marked point. We endow $V_{\bm{\mu}_{n}}$ with the monoidal structure coming from cartesian product (with the componentwise $\bm{\mu}_{n}$-action on it). If $(X,\ *)$ is a finite free pointed $\bm{\mu}_{n}$-set, we will set $\tilde{X}=X\smallsetminus \{*\}$.

\subsubsection{Automorphisms of finite free pointed $\bm{\mu}_n$-sets.} Let $(X,\ *)$ be a finite free pointed $\bm{\mu}_{n}$-set, and let $Aut(X)$ be its automorphism group (as a pointed $\bm{\mu}_{n}$-set). Let $\mathcal{O}_{X}$ be the set of $\bm{\mu}_{n}$-orbits in $\tilde{X}$, and $\Sigma(\mathcal{O}_{X})$ the group of bijective maps from $\mathcal{O}_{X}$ to itself. Every automorphism $f$ of $X$ gives rise to an element $\sigma_{f}\in\Sigma(\mathcal{O}_{X})$; furthermore, choosing representatives $x_{1}, \ldots, x_{t}\in X$ of $\mathcal{O}_{X}$, for $1\leq i\leq t$ we have $f(x_{i})=\mu_{f}(i) \cdot x_{\sigma_f(i)}$ for a unique $\mu_{f} (i)\in \bm{\mu}_{n}$. Setting $\mu(f)=(\mu_{f}(i))_{1\leq i\leq t}\in \bm{\mu}_{n}^{t}$, the datum of $(\mu(f),\ \sigma_{f})$ determines uniquely $f$, hence we obtain an isomorphism $Aut(X) \simeq\bm{\mu}_{n}^{\mathcal{O}_{X}}\rtimes \Sigma(\mathcal{O}_{X})$.

\subsubsection{Abelianisation of $Aut(X)$.} The map $q: Aut(X) \rightarrow \bm{\mu}_{n}\times\Sigma(\mathcal{O}_{X})$ sending $((\mu_{i})_{1 \leq i \leq t},\ \sigma)$ to $(\displaystyle \prod_{i}\mu_{i},\ \sigma)$ is a group homomorphism; in particular, for every $f,\ g\in Aut(X)$ we have $q(gfg^{-1})= q(f)$, hence $q$ does not depend on the choice of $x_{1}, \ldots, x_{t}$. Given $\mu\in \bm{\mu}_{n}$, the elements of the form $((1, 1, \ldots, \mu, 1, \ldots, 1), 1)$ are conjugate in $Aut(X)$. It follows that, if $f: Aut(X) \rightarrow A$ is a morphism with target an abelian group, then $f$ factors through $q$. Therefore the composition of $q$ and the sign map on $\Sigma(\mathcal{O}_{X})$ realises $\bm{\mu}_{n}\times \mathbf{Z}/2\mathbf{Z}$ as the abelianisation of $Aut(X)$.

\begin{lem}\label{trasf.lem}
Let $a\in k^{\times}$ and $m_{a} : k\rightarrow k$ the map sending $x$ to $xa$. Then $m_{a}$ is an automorphism of $k$ as a free pointed $\bm{\mu}_{n}$-set, and the image of $m_{a}$ in $Aut(k)^{ab}$ is $(a^{\frac{q-1}{n}}, \ \varepsilon)$, for some $ \varepsilon \in\{\pm 1\}$.
\end{lem}

\begin{proof}
Choose representatives $x_{1}, \ldots, x_{\frac{q-1}{n}}$ of $\bm{\mu}_{n}$-orbits in $k^{\times}$, and write, for $1\displaystyle \leq i\leq\frac{q-1}{n}$,
\begin{equation*}
x_{i}a=\mu_{i}x_{\sigma_{a}(i)};
\end{equation*}
taking the product of the above equations as $i$ varies we obtain $a^{\frac{q-1}{n}}=\displaystyle \prod_{i=1}^{\frac{q-1}{n}}\mu_{i}$, which implies the lemma in view of the above description of the abelianisation of $Aut(k)$ .
\end{proof}

\begin{rem}
For $n=2$ the previous statement is the classical Gauss's lemma. In general it is a restatement of the (well known) fact that transfer from $k^{\times}$ to $\bm{\mu}_{n}$ equals the $(q-1)/n$-th power map.
\end{rem}

\begin{notat}
If $X$ is a finite free pointed $\bm{\mu}_{n}$-set the composition $Aut(X) \rightarrow Aut(X)^{ab}\rightarrow \bm{\mu}_{n}$ will be denoted by $\Delta_{X}$.
\end{notat}

\begin{lem}
Let $(X,\ *), (Y,\ \star)$ be finite free pointed $\bm{\mu}_{n}$-sets. For every $f\in Aut(X)$ the equality
\begin{equation*}
\Delta_{X\times Y}(f\times Id)=\Delta_{X}(f)
\end{equation*}
holds.
\end{lem}
\begin{proof}
Choose representatives $x_{1}, \ldots, x_{t}$ (resp. $y_{1}, \ldots, y_{u}$) of $\bm{\mu}_{n}$-orbits in $\tilde{X}$ (resp. $\tilde{Y}$). Then we can write
\begin{equation*}
\widetilde{X\times Y}=\left(\coprod_{i=1}^{t}\bm{\mu}_{n}\cdot x_{i}\times\{\star\}\right)\coprod\left(\coprod_{j=1}^{u}\{*\}\times \bm{\mu}_{n}\cdot y_{j}\right)\coprod\left(\coprod_{i,j}\coprod_{\mu\in\bm{\mu}_{n}}\bm{\mu}_{n}\cdot(x_{i},\ \mu\cdot y_{j})\right).
\end{equation*}
Set $g=f\times Id$; we have $g(x_{i},\ \star)=\mu_{f}(i)(x_{\sigma_f(i)}, \star), \ g(*,\ y_{j})=(*,\ y_{j})$ and
\begin{equation*}
g(x_{i},\ \mu\cdot y_{j})=(\mu_{f}(i)x_{\sigma_{f}(i)},\ \mu\cdot y_{j})=\mu_{f}(i)\cdot (x_{\sigma_{f}(i)},\ \mu_{f}(i)^{-1}\mu\cdot y_{j}).
\end{equation*}
Therefore, denoting by $\mu(g)$ be the sequence of roots of unity attached to $g$, we see that for each $i$ such that $\mu_f(i) \neq 1$ the root of unity $\mu_{f}(i)$ contributes $1+un$ times to $\mu(g)$. As $\mu^{1+un}=\mu$ for each $\mu\in \bm{\mu}_{n}$, we find that
\begin{equation*}
\Delta_{X\times Y}(g)=\prod_{\mu\in\mu(g)}\mu=\prod_{i=1}^{t}\mu_{f}(i)=\Delta_{X}(f),
\end{equation*}
as we wanted.
\end{proof}

\begin{corol}
The maps $\Delta_{X}: Aut(X) \rightarrow \bm{\mu}_{n}$ induce a map $\Delta :  K_{1}(V_{\bm{\mu}_{n}})\rightarrow \bm{\mu}_{n}$.
\end{corol}

\subsection{} We have a monoidal functor
\begin{equation*}
V_{k}\rightarrow V_{\bm{\mu}_{n}}
\end{equation*}
sending a $k$-vector space to the underlying pointed $\bm{\mu}_n$-set, and inducing maps $K_{i}(k)\rightarrow K_{i}(V_{\bm{\mu}_{n}})$.

\begin{prop}
The composite $K_{1}(k)\rightarrow K_{1}(V_{\bm{\mu}_{n}})\xrightarrow{\Delta}\bm{\mu}_{n}$ coincides with the $(q-1)/n$-th power map.
\end{prop}
\begin{proof}
For every $r\geq 1$ the composite of the maps $GL_{r}(k)\rightarrow Aut(k^{r})\xrightarrow{\Delta_{k^{r}}}\bm{\mu}_{n}$ must factor through $GL_{r}(k)^{ab}$, hence it sends matrices of determinant one to the identity. Since these maps are compatible as $r$ varies, we see as in the proof of Proposition \ref{legendre} that it suffices to show that the composite $GL_{1}(k)\rightarrow Aut(k)\xrightarrow{p_1}\bm{\mu}_{n}$ sends $a$ to $a^{\frac{q-1}{n}}$; this is the content of Lemma \ref{trasf.lem}.
\end{proof}

\subsection{The case $n=2$.} If $n=2$ the above construction recovers the Legendre symbol as coming from the map in $K$-theory induced by the forgetful functor sending a $k$-vector space to the underlying pointed set with an action of $\{\pm 1\}$. As explained above, the construction given in Section 2 comes instead from the forgetful functor sending a $k$-vector space to the underlying (pointed) set. The reason why the resulting maps on $K_{1}$ coincide is given by the following lemma.

\begin{lem}(cf. \cite[1.2.1]{ks})
Let $X$ be a finite free pointed $\bm{\mu}_{2}$-set and $f: X\rightarrow X$ an automorphism of X. Then the sign of the permutation of $X$ induced by $f$ equals $\Delta_{X}(f)$.
\end{lem}
\begin{proof}
Let $R=\{x_{1}, \ldots, x_{t}\}$ be a set of representatives of the $\bm{\mu}_{2}$-orbits in $\tilde{X}$. For $1\leq i\leq t$ we have $f(x_{i})=\varepsilon_i x_{\sigma_f(i)}$, for a unique permutation $\sigma_{f}$ of the set $\{x_{1},\ldots, x_{t}\}$ and unique signs $\varepsilon_i\in\{\pm 1\}$. For $1\leq i\leq t$ let $\sigma_{i}$ be the transposition exchanging $x_{i}$ with $-x_{i}$. We prove the lemma by induction on the number $n(f)$ of signs $\varepsilon_i$ which equal $-1$; notice that $\Delta_{X}(f)=(-1)^{n(f)}$.

If $n(f)=0$ then $f$ leaves $R$ and $-R$ stable; identifying $R$ with $-R$ via the map sending $x_{i}$ to $-x_{i}$, the map $f$ induces the same permutation $\sigma_{f}$ on the two sets; it follows that $sgn(f)=sgn(\sigma_{f})^{2}=1$. Now assume that $n(f)\geq 1$; pick $i$ such that $f(x_{i})=-x_{\sigma_{f}(i)}$. Then $g=f\circ\sigma_{i}$ sends $x_{i}$ to $x_{\sigma_{f}(i)},$ and $x_{j}$ to $\varepsilon_jx_{\sigma_f(j)}$ for $j\neq i$, hence $n(g)=n(f)-1$ and $\Delta_{X}(f)=-\Delta_{X}(g)$. By induction we have $\Delta_{X}(g)=sgn(g)$; on the other hand $sgn(g)=-sgn(f)$, hence $\Delta_{X}(f)=sgn(f)$.
\end{proof}

\begin{rem}
The outcome of the above discussion is that we have a family of maps
\begin{equation*}
h_{n}^{i} : K_{i}(K)\xrightarrow{\partial_{i-1}} K_{i-1}(k)\rightarrow K_{i-1}(V_{\bm{\mu}_{n}}),\ i\geq 1
\end{equation*}
such that $\Delta\circ h_{n}^{2}$ : $K_{2}(K)\rightarrow\bm{\mu}_{n}$ equals (the map induced by) the symbol $(\cdot, \cdot)_{n}$. If $F$ is a number field containing all the $n$-th roots of unity, we obtain a family of maps $h_{n,v}^{i}: K_{i}(F_{v})\rightarrow K_{i-1}(V_{\bm{\mu}_{n}})$, for $v$ finite place of $F$ lying above a rational prime not dividing $n$; notice that the maps $h_{n,v}^{i}$, as $v$ varies, have the same target. The local symbols at infinite places and places lying above factors of $n$ are however not seen by our construction.
\end{rem}

\section{The tame symbol via determinants and central extensions}

\subsection{Determinant of a finite free pointed $\bm{\mu}_{n}$-set.} Fix an integer $n\geq 1$. As recalled above, for a commutative ring $R$ the determinant maps $\det: GL_{r}(R)\rightarrow R^{\times}, r\geq 1$ induce a surjective map $K_{1}(R)\rightarrow R^{\times}$. Somewhat similarly, if $X$ is a finite free pointed $\bm{\mu}_{n}$-set, we have defined in the previous section a group morphism
\begin{equation*}
\Delta_{X}: Aut(X) \rightarrow \bm{\mu}_{n}
\end{equation*}
inducing a surjection $K_{1}(V_{\bm{\mu}_{n}})\rightarrow \bm{\mu}_{n}$. For $f\in Aut(X)$ we will call $\Delta_{X}(f)$ the {\it determinant} of $f$, and we will denote it from now on by $\det(f)$. It can also be described in terms of the determinant line of $X$, which we define as follows: let $\mathcal{O}_{X}=\{L_{1},\ldots, L_{t}\}$ be the set of $\bm{\mu}_{n}$-orbits in $\tilde{X}$; each $L_{i}$ is isomorphic (non canonically) to $\bm{\mu}_{n}$ as a $\bm{\mu}_{n}$-set; we call such a set a $\bm{\mu}_{n}$-line. Given two $\bm{\mu}_{n}$-lines $L,\ M$ we define their {\it tensor product} $L\otimes M$ as the quotient of the cartesian product $L\times M$ by the equivalence relation $(\mu\cdot l,\ m)\sim(l,\ \mu\cdot m)$ for $\mu\in\bm{\mu}_{n}$. Setting $\mu\cdot(l,\ m)=(\mu\cdot l,\ m)=(l,\ \mu\cdot m)$ makes $L\otimes M$ into a $\bm{\mu}_{n}$-line. Similarly, if $X,\ Y$ are finite free pointed $\bm{\mu}_{n}$-sets such that $\tilde{X}$ and $\tilde{Y}$ are $\bm{\mu}_{n}$-lines then we set $X\otimes Y=\{*\}\coprod (\tilde{X}\otimes\tilde{Y})$. This operation is commutative and associative (up to canonical isomorphism).

We define the {\it determinant line} of the finite free pointed $\bm{\mu}_{n}$-set $X$ as the pointed $\bm{\mu}_{n}$-set
\begin{equation*}
\det(X)=\{*\}\coprod(\otimes_{i=1}^{t}L_{i})
\end{equation*}
where $\mathcal{O}_{X}=\{L_{1},\ldots, L_{t}\}$ and we adopt the convention that the empty tensor product of $\bm{\mu}_{n}$-lines equals $\bm{\mu}_{n}$. Every $f\in Aut(X)$ gives rise to an automorphism of the determinant line $\det(X)$, which is given by multiplication by $\det(f)$.

If $k$ is a finite field of cardinality $n+1$ and $V$ is a finite dimensional $k$-vector space then $V$ has a natural structure of finite free pointed $\bm{\mu}_{n}$-set, and every $k$-linear automorphism $f: V\rightarrow V$ is an automorphism of $V$ as a $\bm{\mu}_{n}$-set; the following lemma ensures that the determinant of $f$ as a $k$-linear map coincides with its determinant as a $\bm{\mu}_{n}$-equivariant map. An alternative proof can be found in \cite[1.2.3]{ks}, where the above definition of determinant of a finite free pointed $\bm{\mu}_{n}$-set was first formulated.

\begin{lem}\label{2dets=1}
Let $k$ be a finite field and $V$ a finite dimensional $k$-vector space. Let $W$ be the tensor product of all the one dimensional $k$-subspaces of $V$. Then the natural map
\begin{equation*}
d :\ Aut_{k}(V)\rightarrow Aut_{k}(W)=k^{\times}
\end{equation*}
is the determinant map.
\end{lem}
\begin{proof}
The map $d: Aut_{k}(V)\rightarrow Aut_{k}(W)=k^{\times}$ factors through the determinant map. Therefore $d$ is a power of the determinant, and to prove the lemma it suffices to check the following statement: fix a basis $(e_{1},\ldots, e_{t})$ of the $k$-vector space $V$ and let $f$ be the automorphism represented in this basis by the diagonal matrix $\mathrm{diag}(g, 1,\ldots, 1)$ where $g$ is a generator of $k^{\times}$; then $d(f)=g$. We will denote by $[a_{1},\ldots, a_{t}]$ the line in $V$ generated by the element with coordinates $(a_{1},\ldots, a_{t})\in k^{t}\smallsetminus \{0\}$ in the chosen basis. The map $f$ acts as multiplication by $g$ on $[1, 0,\ldots, 0]$, and it acts as the identity on every line of the form $[0, a_{2},\ldots, a_{t}]$. If $L=[a_{1},\ldots, a_{t}]$ with $a_{1}\neq 0$ and $a_{i}\neq 0$ for some $2\leq i\leq t$ then we can write $L=\left[\displaystyle \frac{a_{1}}{a_{i}}, \frac{a_{2}}{a_{i}}, \ldots, \frac{a_{t}}{a_{i}}\right]$. Let $m$ be the cardinality of $k^{\times}$; for $1\leq j\leq m-1$ the line $f^{j}(L)=\left[g^{j}\displaystyle \frac{a_{1}}{a_{i}}, \frac{a_{2}}{a_{i}}, \ldots, \frac{a_{t}}{a_{i}}\right]$ is distinct from $L$. This implies that $f$ acts as the identity on $\otimes_{j=0}^{m-1}f^{j}(L)$; it follows that the only non-trivial contribution to $d(f)$ comes from the action of $f$ on $[1, 0,\ldots, 0]$, hence $d(f)=g$.
\end{proof}

\begin{lem}(cf. \cite[2.0.5]{ks})
Let $f: (X,\ *)\rightarrow(Y,\ \star)$ be a map of finite free pointed $\bm{\mu}_{n}$-sets such that $f^{-1}(\star)=*$ and the cardinality of $f^{-1}(y)$ is congruent to one modulo $n$ for every $y\in\tilde{Y}$. Then there is a canonical isomorphism $\det(Y)\xrightarrow{\sim}\det(X)$.
\end{lem}
\begin{proof}
Let $L\subset\tilde{Y}$ be a $\bm{\mu}_{n}$-orbit; the preimage $f^{-1}(L)$ can be written as a disjoint union
\begin{equation*}
f^{-1}(L)=M_{1}\coprod M_{2}\cdots\coprod M_{r}
\end{equation*}
where each $M_{i}$ is a $\bm{\mu}_{n}$-orbit. Furthermore $r\equiv 1\pmod n$. For every $1\leq i\leq r$ the map $f$ induces an isomorphism of $\bm{\mu}_{n}$-sets $f_{i}: M_{i}\rightarrow L$, hence we obtain a $\bm{\mu}_{n}$-equivariant map $L\xrightarrow{(f_{1}^{-1}, \cdot\cdot\cdot, f_{r}^{-1})} \displaystyle \prod_{i=1}^{r}M_{i}$. Since $r\equiv 1 \pmod n$ this map induces an isomorphism of $\bm{\mu}_{n}$-sets $L\simeq\otimes_{i=1}^{r}M_{i}$. Taking the tensor product over all $\bm{\mu}_{n}$-orbits in $\tilde{Y}$ we obtain the desired isomorphism $\det(Y)\simeq\det(X)$ .
\end{proof}

\subsection{Exact sequences.}  Let $Y\hookrightarrow X$ be a $\bm{\mu}_{n}$-equivariant injection of finite free pointed $\bm{\mu}_{n}$-sets. Let $X//Y$ be the set obtained identifying every element of $Y$ with the marked point $*$ of $X$. The induced $\bm{\mu}_{n}$-action makes $X//Y$ a finite free pointed $\bm{\mu}_{n}$-set; the bijection $\{\bm{\mu}_{n}-\text{lines in } \tilde{X}\} = \{\bm{\mu}_{n}-\text{lines in } \tilde{Y}\}\coprod \{\bm{\mu}_{n}-\text{lines in } \widetilde{X//Y}\}$ induces a canonical isomorphism
\begin{equation*}
\det(X)=\det(Y)\otimes\det(X//Y).
\end{equation*}
The following important corollary shows that in suitable circumstances the determinant of finite free pointed $\bm{\mu}_{n}$-sets behaves in exact sequences as its classical linear-algebraic counterpart.

\begin{corol}\label{det-ex-sq}
Let $X, Y, Z$ be finite abelian groups with a $\bm{\mu}_{n}$-action making them finite free pointed $\bm{\mu}_{n}$-sets (marked points being the identity elements). Let $0\rightarrow X\rightarrow Y\rightarrow Z\rightarrow 0$ be an exact sequence of abelian groups with $\bm{\mu}_{n}$-equivariant maps. Then there is a canonical isomorphism
\begin{equation*}
\det(X)\otimes\det(Z)=\det(Y).
\end{equation*}
\end{corol}
\begin{proof}
On the one hand we have a canonical isomorphism $\det(Y)=\det(X)\otimes\det(Y//X)$; on the other hand the induced map $Y//X\rightarrow Z$ satisfies the assumption of the previous lemma because $X$ is a finite free pointed $\bm{\mu}_{n}$-set hence its cardinality is congruent to one modulo $n$.
\end{proof}

\subsection{Duality.} Let $X$ be a pointed $\bm{\mu}_{n}$-set such that $\tilde{X}$ is a $\bm{\mu}_{n}$-line. We define its {\it dual} as $X^{\vee}=\{*\}\coprod \mathrm{Hom}(\tilde{X}, \bm{\mu}_{n})$, where $\mathrm{Hom}$ denotes the set of $\bm{\mu}_{n}$-equivariant maps. The action of an element $\mu \in \bm{\mu}_n$ sends $g:\tilde{X} \rightarrow \bm{\mu_n}$ to the map $x \mapsto \mu \cdot g(x)$. A $\bm{\mu}_{n}$-equivariant isomorphism $f: X\rightarrow Y$ induces a dual isomorphism $f^{\vee}: Y^{\vee}\rightarrow X^{\vee}$. We have a canonical isomorphism
\begin{align*}
\tilde{X}\otimes\widetilde{X^{\vee}}&\rightarrow\bm{\mu}_{n}\\
(v, g)&\mapsto g(v)
\end{align*}
inducing an isomorphism
\begin{equation}\label{dualiso}
X\otimes X^{\vee}\simeq\{*\}\coprod\bm{\mu}_{n}.
\end{equation}

\subsection{Commensurability.} We say that two subgroups $A,\ B$ of an abelian group $V$ are commensurable if the quotients $A/A\cap B$ and $B/A\cap B$ are finite. Commensurability is an equivalence relation on the set of all subgroups of $V$. We will mainly be interested in the case when $V$ is a finite dimensional vector space over a mixed characteristic local field $K$.

\begin{lem}\label{lem-commens}
Let $p$ be a prime, $k\geq 1$ an integer and $A\subset \mathbf{Q}_{p}^{k}$ a subgroup. The following assertions are equivalent:
\begin{enumerate}
\item $A$ is compact and open;
\item $A$ is commensurable with $\mathbf{Z}_{p}^{k}$;
\item $A$ is a free $\mathbf{Z}_{p}$-module of rank $k$.
\end{enumerate}
In particular any two compact open subgroups of $\mathbf{Q}_{p}^{k}$ are commensurable.
\end{lem}
\begin{proof}
Assume that $A$ is compact open; then $A\cap \mathbf{Z}_{p}^{k}$ is open, so $\mathbf{Z}_{p}^{k}/A\cap \mathbf{Z}_{p}^{k}$ is compact and discrete, hence finite. Similarly $A/A\cap \mathbf{Z}_{p}^{k}$ is finite, so $A$ is commensurable with $\mathbf{Z}_{p}^{k}$. Now assume that $A$ is commensurable with $\mathbf{Z}_{p}^{k}$; then $A\cap \mathbf{Z}_{p}^{k}$ has finite index in $\mathbf{Z}_{p}^{k}$ hence it is a free $\mathbf{Z}_{p}$-module of rank $k$. Letting $m$ be the cardinality of $A/A\cap \mathbf{Z}_{p}^{k}$ we have $A\displaystyle \cap \mathbf{Z}_{p}^{k}\subset A\subset\frac{1}{m}(A\cap \mathbf{Z}_{p}^{k})$, therefore $A$ is a free $\mathbf{Z}_{p}$-module of rank $k$. Finally, if $A\subset \mathbf{Q}_{p}^{k}$ is a free $\mathbf{Z}_{p}$-module of rank $k$ then it homeomorphic to $\mathbf{Z}_{p}^{k}$ hence compact open. This proves the equivalence of $(1), (2), (3)$; the last statement follows because commensurability is an equivalence relation.
\end{proof}

\subsection{} We fix from now on a local field $K$ of characteristic zero, with ring of integers $\mathcal{O}$ and residue field $k$ of cardinality $q$. As above, we denote by $v$ the valuation on $K$ and by $\pi$ a uniformiser of $K$. We also fix an integer $n$ dividing $q-1$. Every $\mathcal{O}$-module acquires a natural structure of pointed $\bm{\mu}_{n}$-set via the inclusion $\bm{\mu}_{n}\subset \mathcal{O}$ (the distinguished point being the zero element). Better, we have the following
\begin{lem}\label{O-mod=free}
Let $V$ be a finite $\mathcal{O}$-module; the natural structure of $\bm{\mu}_{n}$-set on $V$ makes it a finite free pointed $\bm{\mu}_{n}$-set.
\end{lem}
\begin{proof}
If $\mu\neq 1\in \bm{\mu}_{n}$ then $\mu-1$ is non-zero modulo $\pi$, hence a unit in $\mathcal{O}$. If $v\in V$ and $\mu v=v$ then $(\mu-1)v=0$, hence $v=0$. Therefore the action of $\bm{\mu}_{n}$ on $V\smallsetminus \{0\}$ is free.
\end{proof}

\subsection{} In view of the previous lemma for every finite $\mathcal{O}$-module $V$ we may consider the determinant $\det(V)$ of $V$ as a finite free pointed $\bm{\mu}_{n}$-set; likewise, any $\mathcal{O}$-linear automorphism $f$ of $V$ is $\bm{\mu}_{n}$-equivariant, hence we can consider its determinant $\det(f)$. Lemma \ref{2dets=1} ensures that this agrees with the usual definition if $V$ is a $k$-vector space {\it and} $n=q-1$ (but not in general for other values of $n$). Corollary \ref{det-ex-sq} yields the following result, which will be the key formal property needed in our arguments.

\begin{corol}\label{det-key}
Let $0\rightarrow X\rightarrow Y\rightarrow Z\rightarrow 0$ be an exact sequence of finite $\mathcal{O}$-modules. Then there is a canonical isomorphism
\begin{equation*}
\det(X)\otimes\det(Z)=\det(Y).
\end{equation*}
\end{corol}

\subsection{}\label{fixdata} Fix from now on an integer $m\geq 1$. We set $K^{m}=V$ and $\mathcal{O}^{m}=V^{+}$. Two $\mathcal{O}$-submodules of $V$ will be called commensurable if they are commensurable as abelian groups. For example, by Lemma \ref{lem-commens} for every $f\in GL_{m}(K)$ the group $f(V^{+})$ is commensurable with $V^{+}$. Our aim is to use the theory of determinants of $\bm{\mu}_{n}$-sets to construct a central extension
\begin{equation}\label{mixed-metapl}
1\rightarrow\bm{\mu}_{n}\rightarrow\widetilde{GL}_{m}(K)\rightarrow GL_{m}(K)\rightarrow 1
\end{equation}
analogous to the one constructed by Arbarello-De Concini-Kac \eqref{eqchar-metapl}, and to relate it to the {\it n}-th power residue symbol.

\begin{defin}
Let $A,\ B\subset V$ be commensurable $\mathcal{O}$-modules. We define
\begin{equation*}
(A\mid B)=\det(A/A\cap B)\otimes\det(B/A\cap B)^{\vee}.
\end{equation*}
\end{defin}

\begin{rem}
We stress that in the definition of $(A \mid B)$ the $\mathcal{O}$-module structure only plays a role to guarantee that the quotients $A/A \cap B$ and $B/A \cap B$ are \emph{free} pointed $\bm{\mu}_n$-sets, thanks to Lemma \ref{O-mod=free}. Given freeness, only the $\bm{\mu}_n$-action is used to obtain the determinants which appear in the definition of $(A\mid B)$. 
\end{rem}

\subsection{Elementary properties.}\label{elem-prop} The symbols $A, \ B$ will denote arbitrary commensurable $\mathcal{O}$-submodules of $V$. We will denote $\{*\}\coprod \bm{\mu}_{n}$ by $\bm{\mu}_{n}^{*}$ and we will use repeatedly the fact that it is canonically isomorphic to its dual, and that $X\otimes\bm{\mu}_{n}^{*}$ is canonically isomorphic to $X$ for every $X$ such that $\tilde{X}$ is a $\bm{\mu}_{n}$-line. We will also denote canonical isomorphisms by an equal sign.

The symbol $(A \mid B)$ enjoys the following properties:
\begin{enumerate}
\item $(A \mid A)=\bm{\mu}_{n}^{*}$.
\item If $B\subset A$ then $(A\mid B)=\det(A/B)$.
\item If $B\supset A$ then $(A\mid B)=\det(B/A)^{\vee}$.
\item $(B\mid A)=(A\mid B)^{\vee}$.
\item $(A\mid B)=(A\mid A\cap B)\otimes(A\cap B\mid B)$.
\end{enumerate}

\subsection{Functoriality.} Let $A, B\subset V$ be commensurable $\mathcal{O}$-modules. Every $f\in GL_{m}(K)$ induces an isomorphism
\begin{equation*}
\rho_{f}:(A\mid B)\rightarrow(f(A)\mid f(B)).
\end{equation*}
Indeed, the determinant of the map $f$ induces and isomorphism $\tau_{f}: \det(A/A\cap B)\rightarrow \det(f(A)/f(A)\cap f(B))$, and the map $f^{-1}$ induces, taking the determinant and passing to duals, an isomorphism $\psi_{f}: \det(B/A\cap B)^{\vee}\rightarrow \det(f(B)/f(A)\cap f(B))^{\vee}$. Hence we can set $\rho_{f}=\tau_{f}\otimes\psi_{f}: (A\mid B)\rightarrow(f(A)\mid f(B))$.

\subsection{The contraction isomorphism.} Together with functoriality, the contraction isomorphism is the second key ingredient needed to construct the central extension \eqref{mixed-metapl}. Let $A, B,  C\subset V$ be commensurable $\mathcal{O}$-submodules. There is a canonical isomorphism, called the {\it contraction isomorphism},
\begin{equation}\label{contraction}
\kappa: (A \mid B)\otimes(B\mid C)\rightarrow(A\mid C).
\end{equation}
The isomorphism $\kappa$ was first constructed in \cite{adck89} in the function field setting; the construction only rests on basic formal properties of the (usual) determinant, the most important of which is given in our situation by Corollary \ref{det-key}. We will reproduce the argument in \cite{adck89} for the reader's convenience; integers in parentheses will refer to the elementary properties listed in \ref{elem-prop}.

\subsubsection{Case $A=C$.}\label{cont.1} In this case $(A\mid A)=\bm{\mu}_{n}^{*}$ by (1) and $(A\mid B)=(B\mid A)^{\vee}$ by (4). The map $\kappa$ is the duality isomorphism \eqref{dualiso}.

\subsubsection{Case $A \supset B \supset C$.}\label{cont.2} We have an exact sequence
\begin{equation*}
0\rightarrow B/C\rightarrow A/C\rightarrow A/B\rightarrow 0
\end{equation*}
hence a canonical isomorphism $\det(A/B)\otimes\det(B/C)\rightarrow\det(A/C)$ , which gives us the desired contraction isomorphism in view of (2).

\subsubsection{Case $A\subset B\subset C$.}\label{cont.3} In this case the natural exact sequence yields an isomorphism $\det(B/A)\otimes \det(C/B)\rightarrow\det(C/A)$, whose dual gives the inverse of $\kappa$.

\subsubsection{General case.} As already mentioned, the following computation is borrowed from \cite{adck89}.

\begin{align*}
(A\mid B)\otimes(B\mid C)\stackrel{(5)}{=}&((A\mid A\cap B)\otimes(A\cap B\mid B))\otimes((B\mid B\cap C)\otimes(B\cap C\mid C))\\
\stackrel{\ref{cont.1}}{=}&(A\mid A\cap B)\otimes(A\cap B\mid A\cap B\cap C)\otimes(A\cap B\cap C\mid A\cap B)\otimes(A\cap B\mid B)\\
\otimes & (B\mid B\cap C)\otimes(B\cap C\mid A\cap B\cap C)\otimes(A\cap B\cap C\mid B\cap C)\otimes(B\cap C\mid C)\\
\stackrel{\ref{cont.2}, \ref{cont.3}}{=}&(A\mid A\cap B\cap C)\otimes(A\cap B\cap C\mid B)\otimes(B\mid A\cap B\cap C)\otimes(A\cap B\cap C\mid C)\\
\stackrel{\ref{cont.1}}{=}&(A \mid A\cap B\cap C)\otimes(A\cap B\cap C\mid C)\\
\stackrel{\ref{cont.2}, \ref{cont.3}}{=}&(A\mid A\cap C)\otimes(A\cap C\mid A\cap B\cap C)\otimes(A\cap B\cap C\mid A\cap C)\otimes(A\cap C\mid C)\\
\stackrel{\ref{cont.1}}{=}&(A\mid A\cap C)\otimes(A\cap C\mid C)\\
\stackrel{(5)}{=}&(A \mid C).
\end{align*}

\subsection{Construction of the extension.} Given $g\in GL_{m}(K)$, the $\mathcal{O}$-modules $V^{+}$ and $g(V^{+})$ are commensurable; furthermore for $f\in GL_{m}(K)$ we have an isomorphism $\rho_{f}: (V^+\mid g(V^{+}))\rightarrow(f(V^{+})\mid fg(V^{+}))$ by functoriality; composing it with the contraction isomorphism we obtain a canonical isomorphism
\begin{equation*}
\iota_{g}^{f}:(V^{+} \mid f(V^{+}))\otimes(V^{+}\mid g(V^{+}))\rightarrow(V^{+}\mid fg(V^{+})).
\end{equation*}
We can now define
\begin{equation*}
\widetilde{GL}_{m}(K)=\{(f,\ s)|f\in GL_{m}(K),\ s\in(V^{+}\mid f(V^{+}))\smallsetminus \{*\}\}
\end{equation*}
with composition law $(f,\ s)\cdot(g,\ t)= (fg,\ \iota_{g}^{f}(s\otimes t))$ making it into a group. We have a group morphism $i: \bm{\mu}_{n}\rightarrow \widetilde{GL}_{m}(K)$ sending $\mu$ to $(Id, \mu)$ and a morphism $p: \widetilde{GL}_{m}(K)\rightarrow GL_{m}(K)$ sending $(f,\ s)$ to $f$; they fit in the sought for central extension
\begin{equation}\label{cent-ext}
1\rightarrow\bm{\mu}_{n} \rightarrow \widetilde{GL}_{m}(K)\rightarrow GL_{m}(K)\rightarrow 1.
\end{equation}

\subsection{The commutator symbol.} Take $f, g\in GL_{m}(K)$ {\it commuting} elements, and choose lifts $\tilde{f},\tilde{g}\in\widetilde{GL}_{m}(K)$ of $f, g$. The commutator $[\tilde{f},\tilde{g}]=\tilde{f}\tilde{g}\tilde{f}^{-1}\tilde{g}^{-1}$ lies in $\bm{\mu}_{n}$ and does not depend on the choice of the lifts. Therefore we can define a symbol
\begin{equation*}
\{f,\ g\}=[\tilde{f},\tilde{g}]\in \bm{\mu}_{n}.
\end{equation*}
The symbol $\{f,\ g\}$ enjoys the following properties:
\begin{enumerate}
\item $\{f,\ g\}=\{g,\ f\}^{-1}$.
\item $\{ff',\ g\}=\{f,\ g\}\{f',\ g\}$.
\item $\{1, g\}=1$ and $\{f^{-1},\ g\}=\{f,\ g\}^{-1}$.
\item If $f, g\in GL_{m}(\mathcal{O})$ then $\{f,\ g\}=1$.
\end{enumerate}
Finally, for $m=1$ we can relate the symbol $\{f,\ g\}$ with the {\it n}-th power residue symbol.

\begin{teo}\label{ext-symb-hilb}
Let $1\rightarrow\bm{\mu}_{n}\rightarrow\tilde{K}^{\times}\rightarrow K^{\times}\rightarrow 1$ be the central extension obtained via the above construction for $m=1$. For every $a,\ b\in K^{\times}$ the following equality holds:
\begin{equation*}
\{a,\ b\}=\left(\frac{a^{v(b)}}{b^{v(a)}}\pmod \pi\right)^{\frac{q-1}{n}}.
\end{equation*}
\end{teo}
\begin{proof}
For $a, b\in K^{\times}$ let us set $\langle a, b\displaystyle \rangle=\left(\frac{a^{v(b)}}{b^{v(a)}}\pmod \pi \right)^{\frac{q-1}{n}}$. Our aim is to check that $\langle a, b\rangle= \{a,\ b\}$. The symbol $\langle\cdot, \cdot\rangle$ satisfies the formal properties (1), (2), (3), (4) of the symbol $\{\cdot, \cdot\}$ listed above. Take $a, b\in K^{\times}$ and write them in the form
\begin{equation*}
a=\pi^{v(a)}u,\ b=\pi^{v(b)}v,\ \ \ u,\ v \in \mathcal{O}^{\times}
\end{equation*}
so that
\begin{equation*}
\{a,\ b\}=\{\pi^{v(a)},\ \pi^{v(b)}\}\{\pi^{v(a)},\ v\}\{u,\ \pi^{v(b)}\}\{u,\ v\}=\{\pi,\ \pi\}^{v(a)v(b)}\{\pi,\ v\}^{v(a)}\{u,\ \pi\}^{v(b)}\{u,\ v\}.
\end{equation*}
The same equality holds with $\langle\cdot, \cdot\rangle$ in place of $\{\cdot, \cdot\}$, hence it suffices to show that
\begin{enumerate}
\item $\{\pi,\ \pi\}=1$.
\item $\{u,\ v\}=1$.
\item $\{u,\ \pi\}=u^{\frac{q-1}{n}}\pmod \pi$. 
\end{enumerate}

The first property holds because, for any lift $\tilde{\pi}$ of $\pi$, we have $\{\pi,\ \pi\}=[\tilde{\pi},\tilde{\pi}]=1$; the second property is a special case of (4) above. It remains to prove that, for every $u\in \mathcal{O}^{\times},$
$$
\{u,\ \pi\}=u^{\frac{q-1}{n}}\ (\mathrm{m}\mathrm{o}\mathrm{d}\ \pi)\ .
$$
Pick $s\in(\mathcal{O}\mid u\mathcal{O})\smallsetminus \{*\}$ and $t\in(\mathcal{O}\mid \pi \mathcal{O})\smallsetminus \{*\}$, and consider the lifts $\tilde{u}=(u,\ s)$ and $\tilde{\pi}=(\pi,\ t)$ of $u$ and $\pi$ respectively. We have $(\mathcal{O}\mid \pi \mathcal{O})=\det(\mathcal{O}/(\pi)),\ (\mathcal{O}\mid u\mathcal{O})=\bm{\mu}_{n}^{*}$ and the isomorphism
\begin{equation*}
\iota_{\pi}^{u}\ :\ \bm{\mu}_{n}^{*}\otimes\det(\mathcal{O}/(\pi))\rightarrow\det(\mathcal{O}/(u\pi))
\end{equation*}
sends $s\otimes t$ to $\det(u)st$, where $\det(u)$ denotes the determinant of the multiplication by $u$ map on $\mathcal{O}/\pi$.

On the other hand the map $\rho_{\pi}: (\mathcal{O}\mid u\mathcal{O})\rightarrow(\pi \mathcal{O}\mid \pi u\mathcal{O})$ is the identity, therefore $\iota_{u}^{\pi}$ sends $t\otimes s$ to $st$. Therefore we obtain $[\tilde{u},\tilde{\pi}]=\det(u)=u^{\frac{q-1}{n}}\pmod \pi$, where the last equality follows from Lemma \ref{trasf.lem}.
\end{proof}

\subsection{Relative dimension.} The symbol $\{\cdot, \cdot\}$ obtained via the central extension \eqref{cent-ext} for $m=1$ equals the Hilbert symbol only up to sign. Examining the above proof one sees that the issue is that $\{\pi,\ \pi\}$ always equals 1. This is inevitable with our definition, because $\{a,\ a\}=1$ for every $a\in K^{\times}$, as any lift of $a$ commutes with itself. For the same reason, this discrepancy also appears in \cite{adck89} and \cite{pr02}. In \cite{adck89} the authors modify their symbol, which they denote by $\{g,\ h\}_{A}^{V}$, by a sign taking into account the ``relative dimension'' of $A$ and its image via $f$ and $g$ (where $A$ is the counterpart of our $V^{+}$). They then show that such a modification of the symbol makes it equal to the tame symbol \cite[Lemma 6.1]{adck89}. We will end this document explaining how the notion of dimension of a finite free pointed $\bm{\mu}_{n}$-set, introduced in \cite{ks}, can be used to modify our symbol $\{\cdot, \cdot\}$ in a way similar to \cite{adck89}; if $q$ is odd, we will show that the resulting symbol equals the {\it n}-th power residue symbol.

\subsection{} Let $(X,\ *)$ be a finite free pointed $\bm{\mu}_{n}$-set; the {\it dimension} of $X$, denoted by $\dim X$, is the number of $\bm{\mu}_{n}$ orbits in $X\smallsetminus \{*\}$.

As in \ref{fixdata}, let $V^{+}=\mathcal{O}^{m}$, and let $A, B\subset V=K^m$ be commensurable $\mathcal{O}$-modules. Fix a positive integer $n \mid q-1$ and consider $A, B$ as $\bm{\mu}_{n}$-sets. We define
\begin{equation*}
[A\mid B]=\dim(A/A\cap B)-\dim(B/A\cap B).
\end{equation*}
The following is an additive analogue of Corollary \ref{det-key}.

\begin{lem}\label{dim-exsq}
Let $0\rightarrow X\rightarrow Y\rightarrow Z\rightarrow 0$ be an exact sequence of finite $\mathcal{O}$-modules. Let $m$ be a divisor of $q-1$. Then $\dim(X)+\dim(Z)\equiv\dim(Y)\pmod m$.
\end{lem}
\begin{proof}
Let $q^{x}$ (resp. $q^{y}$, resp. $q^{z}$) be the cardinality of $X$ (resp. $Y$, resp. $Z$). Then we find
\begin{equation*}
\dim(Y)=\frac{q^{y}-1}{n}=\frac{q-1}{n}(1+q+\ldots +q^{y-1})\equiv y\frac{q-1}{n}\pmod m.
\end{equation*}
where the congruence holds because $q\equiv 1 \pmod m$ . The same formula holds for the dimension of $X$ and $Z$, hence the lemma follows from the fact that $y=x+z$.
\end{proof}

\subsection{} Let $\{\cdot, \cdot\}: K^{\times}\times K^{\times}\rightarrow \bm{\mu}_{n}$ be the symbol considered in Theorem \ref{ext-symb-hilb}. Inspired by \cite{adck89} we define, for $a,\ b\in K^{\times}$:
\begin{equation*}
\langle a,\ b\rangle=(-1)^{[\mathcal{O}\mid a\mathcal{O}][\mathcal{O}\mid b\mathcal{O}]}\{a,\ b\}.
\end{equation*}

\begin{corol}
If $q$ is odd then for every $a, b\in K^{\times}$
\begin{equation*}
\langle a,\ b\rangle=\left((-1)^{v(a)v(b)}\frac{a^{v(b)}}{b^{v(a)}}\pmod \pi\right)^{\frac{q-1}{n}}.
\end{equation*}
\end{corol}
\begin{proof}
For every $a\in K^{\times}$ we have
\begin{equation}
[\mathcal{O}\mid a\mathcal{O}]\equiv\frac{q-1}{n}v(a)\pmod 2.
\end{equation}
Indeed it suffices to check this for $v(a)$ positive; if $v(a)=1$ the congruence holds true by definition, and the general case follows by induction using Lemma \ref{dim-exsq} with $m=2$.

Now, for $a, b\in K^{\times}$ we find
\begin{equation*}
[\mathcal{O}\mid a\mathcal{O}][\mathcal{O}\mid b\mathcal{O}]\equiv\frac{q-1}{n}v(a)\frac{q-1}{n}v(b)\equiv\frac{q-1}{n}v(a)v(b)\pmod 2
\end{equation*}
hence the result follows from Theorem \ref{ext-symb-hilb}.
\end{proof}

\bibliographystyle{amsalpha}
\bibliography{bibKhilb}

\Addr

\end{document}